\documentclass{amsart}
\usepackage{amsthm}
\usepackage{amssymb}
\usepackage{mathtools}

\usepackage{enumitem}
\setlist[enumerate,1]{label=\arabic*., leftmargin=18pt, itemsep=0pt, topsep=0pt}

\usepackage{graphicx}
\usepackage{tikz}
\usetikzlibrary{arrows,cd}

\definecolor{lightgrey}{rgb}{.804,.804,.756}

\usepackage[pdfauthor={V. Lebed},
pdftitle={Conjugation groups and structure groups of quandles},%
colorlinks=false,linkbordercolor=lightgrey,citebordercolor=lightgrey,urlbordercolor=lightgrey]{hyperref}

\definecolor{myred}{rgb}{.545,0,0}
\definecolor{myblue}{rgb}{.024,.15,.645}    
\definecolor{mygreen}{rgb}{0,.455,0}

\newcommand{\N}{\mathbb{N}}
\newcommand{\Z}{\mathbb{Z}}

\newcommand{\Ker}{\operatorname{Ker}}
\renewcommand{\Im}{\operatorname{Im}}

\newcommand{\Inn}{\operatorname{Inn}}
\newcommand{\Orb}{\operatorname{Orb}}
\newcommand{\OO}{\operatorname{\mathcal{O}}}

\newcommand{\Al}{\operatorname{Al}}
\newcommand{\Conj}{\operatorname{Conj}}
\newcommand{\As}{\operatorname{As}}
\newcommand{\oAs}{\operatorname{\overline{As}}}

\newcommand{\Stab}{\operatorname{Stab}}
\newcommand{\Transp}{T}
\newcommand{\cb}{$\overline{C}$}
\newcommand{\cc}{$C$}
\newcommand{\CC}{\mathfrak{C}}
\newcommand{\ff}{\operatorname{f}}
\newcommand{\Ab}{\operatorname{Ab}}
\newcommand{\Abf}{\Ab_{\ff}}
\newcommand{\Gf}{G_{\ff}}
\newcommand{\Tor}{\operatorname{Tor}}
\newcommand{\opi}{\operatorname{\overline{\pi}}}
\newcommand{\Art}{\mathcal{A}}
\newcommand{\Deh}{\mathcal{D}}
\newcommand{\fg}{\operatorname{fg}}
\renewcommand{\ng}{\operatorname{ng}}
\newcommand{\ig}{\operatorname{ig}}
\renewcommand{\gg}{\operatorname{g}}
\renewcommand{\lg}{\operatorname{lg}}
\newcommand{\Supp}{\operatorname{Supp}}
\newcommand{\RSupp}{\operatorname{RSupp}}
\newcommand{\RR}{\operatorname{R}}

\newcommand\op{\mathrel{\triangleleft}}

\newcommand{\twoheadrightarrowdown}{%
\mathrel{\reflectbox{\rotatebox[origin=c]{-90}{$\twoheadrightarrow$}}}}

\makeatletter
	\theoremstyle{plain}
\newtheorem{thm}{Theorem}[section]
\newtheorem{lmm}[thm]{Lemma}
\newtheorem{crl}[thm]{Corollary}
\newtheorem{prp}[thm]{Proposition}

	\theoremstyle{definition}
\newtheorem{dfn}[thm]{Definition}

\newtheorem{xmp}[thm]{Example}
	\theoremstyle{remark}
\newtheorem{rmr}[thm]{Remark}

\newtheorem{letterthm}{Theorem}

\setlist{nolistsep}
\setenumerate{topsep=0pt}
\makeatother

\begin{document}

\title{Conjugation groups \\ and structure groups of quandles}

\begin{abstract} 
Quandles are certain algebraic structures showing up in different mathematical contexts. A group $G$ with the conjugation operation forms a quandle, $\operatorname{Conj}(G)$. In the opposite direction, one can construct a group $\operatorname{As}(Q)$ starting from any quandle $Q$. These groups are useful in practice, but hard to compute. We explore the group $\operatorname{As}(\operatorname{Conj}(G))$ for so-called \emph{$\overline{C}$-groups} $G$. These are groups admitting a presentation with only conjugation and power relations. Symmetric groups $S_n$ are typical examples. We show that for $\overline{C}$-groups, $\operatorname{As}(\operatorname{Conj}(G))$ injects into $G \times  \mathbb{Z}^m$, where $m$ is the number of conjugacy classes of $G$. From this we deduce information about the torsion, center, and derived group of $\operatorname{As}(\operatorname{Conj}(G))$. As an application, we compute the second quandle homology group of $\operatorname{Conj}(S_n)$ for all $n$, and unveil rich torsion therein.
\end{abstract}



	\author[Victoria Lebed]{Victoria Lebed}
	\address{Normandie Univ, UNICAEN, CNRS, LMNO, 14000 Caen, France}
	\email{victoria.lebed@unicaen.fr}

\maketitle

\section{Introduction}\label{Sec:Intro}

A \emph{quandle} is a set $Q$ endowed with an idempotent binary operation~$\op$ for which all right translations are quandle automorphisms. In other words, it satisfies the following axioms for all $a,b,c \in Q$:
\begin{enumerate}
\item idempotence: $a \op a = a$;
\item bijectivity of the right translations $R_c \colon Q \to Q,\ a \mapsto a \op c$;
\item self-distributivity: $(a \op b) \op c = (a \op c) \op (b \op c)$.
\end{enumerate}
Quandles and generalisations thereof have remarkably wide ranging applications: construction of invariants of knots and knotted surfaces (the three quandle axioms encoding the three Reidemeister moves), classification of Hopf algebras, construction and study of set-theoretic solutions to the Yang--Baxter equation, integration of Leibniz algebras (non-symmetric generalisations of Lie algebras). See the surveys \cite{AndrGr, Kinyon, CarterSurvey, ElhamdadiNelsonBook, KamadaSurvey, LebedSDYBESurvey} for more detail. Recently, the relevance of quandles in theoretical physics was brought to light as well \cite{SDPhysics,pracks}. 

Quandles can be constructed starting from mathematical structures of very different nature. The two fundamental examples are:
\begin{enumerate}
\item \emph{Alexander quandles} $\Al(M)$, i.e., $\Z[t^{\pm 1}]$-modules $M$ with 
\[a \op b = ta+(1-t)b,\]
\item \emph{conjugation quandles} $\Conj(G)$, i.e., groups $G$ with the conjugation operation 
\[a\op b = b^{-1}ab.\]
\end{enumerate}

The \emph{structure group} (also called the \emph{associated}, \emph{enveloping}, or \emph{adjoint group}) of a quandle $(Q,\op)$ is defined by the following quadratic presentation:
\begin{equation}\label{Eq:StrGroups} 
\As(Q) \ = \ \langle \ e_a, \, a \in Q \,|\, e_a e_b = e_b e_{a\op b}, \, a,b \in Q \ \rangle.
\end{equation}
It can be considered from different viewpoints. 

First, it is a deformation of a free abelian group, and inherits nice properties thereof. In particular, for finite $Q$, $\As(Q)$ is virtually free abelian: carefully chosen powers of its generators $e_a^{d_a}$ form a central free abelian subgroup with a \emph{Coxeter-like finite quotient} $\oAs(Q)$ \cite{LebVen3}. This gives a rich source of nicely behaving groups. 

Second, the structure group construction yields a functor from the category of quandles to that of groups, which is left adjoint to the conjugation quandle functor.

Third, $\As(Q)$ is an invariant of the quandle $Q$, which opens quandle theory to group-theoretic tools. In particular, a sufficient understanding of $\As(Q)$ allows one to compute the second \emph{quandle homology} group $H_2(Q)$, which is another quandle invariant crucial for applications and in general difficult to compute (see the above referenced surveys for detail). This can be done via Eisermann's formula \cite{Eisermann} 
\begin{equation}\label{E:Eisermann} 
 H_2(Q)\ \cong\ \bigoplus_{\OO \in \Orb(Q)} \left(\Stab_{\As(Q)}(a_{\OO}) \cap \Ker (\varepsilon)\right)_{Ab}.
\end{equation}
Here $\Orb(Q)$ is the set of the \emph{orbits} of $Q$ with respect to the action of the \emph{inner group} $\Inn(Q)$, which is the group of permutations of $Q$ generated by the right translations $R_c$; $a_{\OO} \in Q$ is any representative of the the orbit $\OO$; $\As(Q)$ acts on $Q$ via $a \cdot e_b = a \op b$; and $\varepsilon$ is the \emph{degree map}, which is the group morphism defined by
\begin{align}
\varepsilon \colon \As(Q) &\to \Z,\label{E:Degree}\\
e_a &\mapsto 1.\notag
\end{align}
The free part of $H_2(Q)$ is known to be $\Z^{c(Q)(c(Q)-1)}$ \cite{EtGrRackCohom,LiNel}. Here $c(Q)$ is the number of the $\Inn(Q)$-orbits of $Q$; for $Q=\Conj(G)$, this is precisely the number of conjugacy classes of $G$. It is the torsion part that is used in applications, and it can be highly non-trivial.

Structure groups are very difficult to compute in practice. We know of several notable examples only: an explicit description of the groups $\As(\Al(M))$ for connected (with respect to the $\Inn(\Al(M))$ action) Alexander quandles \cite{ClauwensAl}, resulting in a computation of $H_2$ for some of them \cite{GarIglVen}; a classification of all quandles with free abelian structure groups, together with partial results on the structure groups and homology of abelian quandles (having commuting right translations) \cite{LebMorAbelian}; and full or partial results for certain more involved classes of quandles in \cite{NosakaAdjoint}. Note that in all these cases, $H_2$ does have torsion in general. 

\medskip 
The original aim of this work was to compute the structure group and the second homology group of the conjugation quandle of a symmetric group $S_n$. The tools we developed turned out to apply to a much wider class of groups, which we call \emph{\cb-groups}. These are groups admitting a presentation with only conjugation relations ($b^{-1}ab=c$) and power relations ($a^k=1, \ k \geq 2$). Without loosing generality, we allow only one power relation per conjugacy class. If only conjugation relations appear, we talk about \emph{conjugation groups}, or briefly \emph{\cc-groups}. Among \cb-groups one finds free and free abelian groups; Artin and Coxeter groups, and various generalisations thereof (virtual, welded, homotopy, etc.) \cite{VirtualArtin, DarneHomotopyWelded}; cactus and twin groups, and various generalisations thereof \cite{VirtualTwin, mostovoy2023roundtwingroupsstrands, chemin2024minimalpresentationfinitequotients}; Thompson's group $F$ \cite{ChouraquiThompson,SzymikThompson}; knot groups; structure groups and their finite quotients; and many more. Indeed, the commutation relation, for instance, can be interpreted as $b^{-1}ab=a$, the braid relation $aba=bab$ as $b^{-1}ab=c$ and $a^{-1}ca=b$ ($c$ being an extra generator), etc.
 
\cc-groups appeared in \cite{KulikovCGroups,KulikovCGroupsGeom}, where it was shown that, for any $n \geq 2$, being a \cc-group is the same as being the fundamental group of the complement of an $n$-dimensional compact orientable manifold without boundary in $\mathbb{S}^{n+2}$. We borrow this term, even if it has an alternative usage in group theory. Note that these groups independently appeared in \cite{LOG} in a different context and under a different name, that of \emph{LOG-groups} (LOG = labelled oriented graph). 

There are two easy ways to extract some partial information from an element of the structure group 
\[A(G) \ := \ \As(\Conj(G)).\] 
 The first one is to abelianise. Relation~\eqref{Eq:StrGroups} becomes then $e_a = e_{a\op b}$, thus all generators from the same orbit become undistinguishable. Hence a map 
\[\Ab \colon A(G) \twoheadrightarrow \Z^{\CC(G)},\] 
where $\CC(G)$ is the set of \emph{conjugacy classes} of $G$. This map works for all quandles. The second one is proper to conjugation quandles: one can return to the original group $G$ via the group morphism defined by
\begin{align*}
\pi \colon A(G) &\twoheadrightarrow G,\\
e_a &\mapsto a.
\end{align*}
It turns out that for a \cb-group $G$, the combination of these two simple maps suffices to understand the entire group $A(G)$:
\begin{letterthm}\label{T:Inj}
For a \cb-group $G$, the following map is injective:
\[\pi \times \Ab \colon A(G) \to G \times \Z^{\CC(G)}.\]
\end{letterthm}
The image of this map is described in terms of a pullback in Theorem \ref{T:Pullback}.

As a result, various properties of the group $A(G)$ can be deduced from those of the \cb-group $G$, as long as one understands the conjugacy classes of $G$ reasonably well. Here are only some useful corollaries. To announce some of them, we will need to fix a \cb-type presentation $\Pi$ of $G$. These properties will be made explicit for our running example $G=S_n$, with its classical presentation.
\begin{enumerate}
\item The kernel of the map $\pi$ is a subgroup of $\Z^{\CC(G)}$, hence free abelian. A more refined analysis (Proposition \ref{P:Kernel}) yields an isomorphism 
\[K(G) \ := \ \Ker(\pi) \ \cong\ \Z^{\CC'(G)},\] 
where $\CC'(G)$ is the set of conjugacy classes of $G$ containing no infinite-order generators from $\Pi$. Thus $A(G)$ is a central extension of $G$. 
 We study the properties of the associated group $2$-cocycle from $H^2(G,\Z^{\CC'(G)})$ (Propositions \ref{P:2cocycle} 
 and \ref{P:2cocycleImage}). In particular, this yields a short exact sequence
\[0 \ \to \ \Z^{P(n)}\ \to \ A(S_n) \ \overset{\pi}{\to} \ S_n \ \to \ 0,\] 
$P(n)$ being the \emph{partition number} of $n$, which is the number of conjugacy classes of $S_n$.
\item The map $\pi$ restricts to an isomorphism of the derived groups:
\[A(G)' \ \cong\ G'.\]
\item The torsion of $A(G)$ is $\Ker(\Ab) \cap \pi^{-1}(\Tor(G))$, which is isomorphic to a subgroup of $\Tor(G)$. For a \textbf{finite} \cb-group $G$, $\Tor(G)=G$, hence 
\[\Tor(A(G)) \ =\ \Ker(\Ab) \ =\ A(G)' \ \cong\ G'.\] 
For $S_n$, one obtains the alternating group:
\[\Tor(A(S_n)) \ =\ A(S_n)'\ \cong\ S_n' \ = \ A_n.\]
In particular, contrary to \cc-groups, $S_n$ is not a subgroup of $A(S_n)$.
\item The center of $A(G)$ is
\[Z(A(G)) \ = \ \pi^{-1}(Z(G)).\]
For a centerless group, this is a subgroup of $\Z^{\CC(G)}$, hence free abelian. In particular, for $n>2$, one gets
\[Z(A(S_n)) \ = \ K(S_n) \ \cong\ \Z^{P(n)}.\] 
\item An element $g \in A(G)$ acts trivially on $a \in G$ if and only if $\pi(g)$ and $a$ commute in $G$. Thus, 
\[\Stab_{A(G)}(a) \ = \ \pi^{-1}(\Stab_G(a)).\]
We will develop this idea further in Proposition \ref{P:H2Cgroup}. Since commutation in $S_n$ is well understood, we will use this to compute in Theorem \ref{T:H2Sn} the torsion of $H_2(\Conj(S_n))$, which turns out to be very rich. 
\end{enumerate}

\medskip
While the above subgroup interpretation of $A(G)$ is instrumental in studying its properties, an alternative direct product type description is better adapted for other purposes. To get to that, we first recall a somewhat surprising characterisation of \cc-groups and certain \cb-groups: 

\begin{letterthm}\label{T:CGroupIFF}
\begin{enumerate}
\item \cite{SinghEtAlDehnQuandles} A group $G$ is a \cc-group if and only if $G \cong \As(Q)$ for some quandle $Q$. 
\item \cite{AkitaEtAlFiniteQuotient} A finite centerless group $G$ is a \cb-group if and only if $G \cong \oAs(Q)$ for some finite quandle $Q$. 
\end{enumerate}
\end{letterthm}

We give short proofs of these results with our methods.

Note that the isomorphism $G \cong \As(Q)$ or $\oAs(Q)$ yields a new presentation of $G$ with \cc-type relations for \textbf{all} instead of \textbf{some} pairs of generators only.

In \cite{SzymikPowers} the theorem was extended to \emph{pq-generated groups} $G$, which allow conjugation and generalised power relations ($a^k=b^m,\ k,m\in \Z$) in their presentation, and to appropriate quotients of the structure group $\As(\Conj(G))$. Those quotients were studied in detail there. 

A quandle $Q$ satisfying the theorem is constructed explicitly, starting from a \cb-type presentation $\Pi$ of $G$. Indeed, one takes the union $\Gamma^G$ of $G$-conjugacy classes of the generator set $\Gamma$ of $\Pi$. This is reminiscent of dual braid monoids \cite{BKLDual,BessisDual}, generated by all ``positive'' conjugates of the classical generators of braid groups. The quandle $\Gamma^G$ is a particular case of \emph{Dehn quandles} from \cite{SinghEtAlDehnQuandles}. When $G$ is Thompson's group $F$, we can reformulate \cite[Prop. 6.4]{SzymikThompson} to say that Thompson's quandle, described by generators and relations in \cite{SzymikThompson}, is isomorphic to our $\Gamma^F$.

For our favourite groups $S_n$ with their classical presentation, $\Gamma^{S_n}$  is the set $\Transp_n$ of all transpositions from $S_n$, endowed with the conjugation operation. We thus recover the following isomorphism from \cite{GarIglVen}:
\[S_n \ \cong  \ \oAs(\Transp_n).\] 

Our next result shows that $A(G)$ is a trivial extension of the structure group $G^{\Deh}:=\As(\Gamma^G)$ of the much smaller quandle $\Gamma^G$:
\begin{letterthm}\label{T:DirectProduct}
For any \cb-group $G$, one has a group isomorphism
\[A(G) \ \cong \ G^{\Deh} \times \Z^{\CC_{\ng}(G)},\]
where $\CC_{\ng}(G)$ is the set of conjugacy classes of $G$ containing no generators from $\Pi$. 
\end{letterthm}

We propose several alternative descriptions of $G^{\Deh}$, adapted to different purposes.

First, it is an intermediate object between $G$ and the \emph{Artin-type lift} $G^{\Art}$ obtained from a presentation $\Pi$ of $G$ by forgetting all power relations:
\begin{letterthm}\label{T:DehnAsIntermediate}
For a group $G$ with a \cb-type presentation $\Pi$, one has a sequence 
\[ G^{\Art} \ \twoheadrightarrow \ G^{\Deh} \ \twoheadrightarrow \ G\]
of group surjections, where the first map imposes the element $a^k$ to be central whenever $\Pi$ contains a relation $a^k=1$; and the second map collapses all these central elements, presenting $G^{\Deh}$ as a central extension of $G$:
\begin{equation}\label{E:GDvsG}
0 \ \to \ \Z^{\CC_{\fg}(G)} \ \to \ G^{\Deh} \ \to \ G \ \to \ 0. 
\end{equation}
Here $\CC_{\fg}(G)$ is the set of all conjugacy classes of $G$ containing finite order generators. 
\end{letterthm}
In particular,
\begin{itemize}
\item for a \cc-group $G$, the three groups above coincide, and we recover the isomorphism $A(G) \cong G \times \Z^{\CC_{\ng}(G)}$ from \cite{Ryder}, where it was conjectured to characterise \cc-groups;
\item for Coxeter groups, these surjections were studied in \cite{AkitaCoxeter} (with the notations $W=G,\ A_W=G^{\Art},\ Q_W= \Gamma^G$), generalising the $S_n$ case from \cite{NicholsSimpleRacks}.
\end{itemize}    

Second, we inject $G^{\Deh}$ into $G \times \Z^{\CC_{\fg}(G)}$ in Theorem \ref{T:InjForDehn}, and interpret this as a pullback in Proposition \ref{P:PullbackForDehn}.

Finally, when all generators of $G$ are from the same conjugacy class, Proposition \ref{P:SemiDirect} presents $G^{\Deh}$ as a semi-direct product
\[G^{\Deh} \ \cong \ G' \rtimes \Z.\]

For $S_n$, one has
\[ S_n^{\Art} \cong B_n \ \twoheadrightarrow \ S_n^{\Deh} \cong  \raisebox{.2em}{$B_n \times \langle t \rangle$}\left/\raisebox{-.2em}{$\sigma_1^2=t$}\right.  \ \twoheadrightarrow \ S_n \cong \raisebox{.2em}{$B_n$}\left/\raisebox{-.2em}{$\sigma_1^2=1$}\right.,\]
 and Theorem \ref{T:DirectProduct} yields 
\[A(S_n) \ \cong \ S_n^{\Deh} \times \Z^{P(n)-1}.\]
Further, $S_n^{\Deh}$ is a central extension of $S_n$, 
\[S_n^{\Deh} \ \cong\ S_n \underset{\varphi}{\times} \Z,\] 
for the $2$-cocycle
\begin{align*}
\varphi \colon \hspace*{1cm} S_n^2 &\to \Z,\\
(\alpha,\beta) &\mapsto (-\lg(\alpha\beta)+\lg(\alpha)+\lg(\beta))/2,
\end{align*} 
where $\lg$ is the minimal number of transpositions needed to write a permutation. Finally, we recover the following description from \cite{Eisermann, AkitaCoxeter}:
\[ S_n^{\Deh} \ \cong \ A_n \rtimes \Z,\]
where $1 \in \Z$ acts on $A_n$ by conjugation by, say, the generator $\sigma_1$.

\section{Structure groups as subgroups}\label{Sec:Subgroup}

The aim of this section is to prove Theorem~\ref{T:Inj}, that is, to present the structure group $A(G)$ as a subgroup of the much more manageable direct product $G \times \Z^{\CC(G)}$.

Fix a group $G$ with a \cb-type presentation $\Pi$. Denote its set of generators by $\Gamma$. We will need several notations, relevant for all subsequent sections. 

First, let us split the set of the conjugacy classes of $G$ as 
\[\CC(G) = \CC_{\gg}(G) \sqcup \CC_{\ng}(G), \hspace*{1.5cm} \CC_{\gg}(G) = \CC_{\fg}(G) \sqcup \CC_{\ig}(G),\]
 where the classes from $\CC_{\gg}(G)$ contain some \underline{g}enerator $a \in \Gamma$, while those from $\CC_{\ng}(G)$ do \underline{n}ot; and in a class $\OO \in \CC_{\fg}(G)$ all elements are of \underline{f}inite order, denoted by $k(\OO)$, while in $\CC_{\ig}(G)$ everyone is of \underline{i}nfinite order. Recall that we allow only one power relation per conjugacy class, so $k(\OO)$ is exactly the exponent in the only power relation $a^k=1$ with $a \in \OO$.
 
From the presentation $\Pi$, one directly deduces the abelianisation of $G$:
\begin{equation}\label{E:AbG}
\Ab(G) \ \cong \ \Z^{\CC_{\ig}(G)} \oplus \Gf, \hspace*{1cm} \text{ where } \ \Gf :=\  \bigoplus_{\OO \in \CC_{\fg}(G)} \Z_{k(\OO)}.
\end{equation}

Further notations are needed to make the abelianisation map for the structure group, $\Ab \colon A(G) \twoheadrightarrow \Z^{\CC(G)}$, more explicit. Denote by $c_{\OO}$ the generator of $\Z^{\CC(G)}$ corresponding to $\OO \in \CC(G)$; thus $\Ab(e_a)=c_{\OO}$ for $a \in \OO$. Also, denote by $\varepsilon_{\OO} \colon A(G) \twoheadrightarrow \Z$ the composition of $\Ab$ with the projection onto the corresponding factor of $\Z^{\CC(G)}$; thus $\varepsilon_{\OO}(e_a)=1$ if $a \in \OO$, and $0$ if $a \notin \OO$. Observe that $\Ab = \bigoplus_{\OO \in \CC(G)} \varepsilon_{\OO}$, and the degree map is just the sum $\varepsilon= \sum_{\OO \in \CC(G)} \varepsilon_{\OO}$.

The group $A(G)$ has a very big center. This will be discussed in more detail later; now we will need only central elements of the form $e_a^k$. To study those, let us recall a very general lemma, rediscovered by a remarkable number of authors:
\begin{lmm}\label{L:CentralPowers}
If a power $a^k$ is central in a (not necessarily \cb-) group $G$, then the power $e_a^k$ is central in $A(G)$, and one has $e_a^k = e_b^k$ for any $b$ from the same conjugacy class. Taking at most one such power per conjugacy class, one generates a central free abelian subgroup of $A(G)$.  
\end{lmm}
Note that such central powers exist in the structure group of any finite quandle $Q$, and are used to construct the finite quotients $\oAs(Q)$.
\begin{proof}
The centrality statement follows from the computation
\[e_c \op (e_a^k) = e_{c \op (a^k)} =a_c\]
for all $c \in G$. Here and below the symbol $\op$ in a group always stands for the conjugation operation. This implies 
\[ e_a^k = e_a^k \op e_c = (e_a \op e_c)^k=e_{a \op c}^k, \]
hence $e_a^k = e_b^k$ for any $b$ from the same conjugacy class as $a$. Finally, the freeness follows from the freeness of the image in the abelianisation, since 
\[\Ab(e_a^k)=kc_{\OO} \in \bigoplus_{\OO \in \CC(G)} \Z c_{\OO}=\Z^{\CC(G)}. \qedhere\]
\end{proof}

For a \cb-group $G$, to any $\OO \in \CC_{\fg}$ one then associates the central element
\begin{equation}\label{E:T}
t_{\OO} = e_{a}^{k(\OO)}, \ a \in \OO,
\end{equation}
which is independent of $a \in \OO$. Indeed, the element $a^{k(\OO)}=1$ is central in $G$.

\begin{xmp}\label{Ex:SnPresentation}
To help the reader better digest these notations, we write them out explicitly for the symmetric group $S_n$ with its classical presentation, enhanced by extra generators to get to the \cb-type:
\begin{equation}\label{E:PresentationSn}
S_n = \Biggl\langle      
            \begin{array}{ll|l}
            & &  \sigma_i \op \sigma_j = \sigma_i,\ i<j-1,\\
            \sigma_i,& 1 \leq i < n, & \sigma_i \op \sigma_{i+1} =\sigma_{i,i+1}, \\
            \sigma_{i,i+1},& 1 \leq i < n-1, & \sigma_{i,i+1} \op \sigma_{i} = \sigma_{i+1},\\
            & &  \sigma_i^2=1
        \end{array}
     \Biggr\rangle.
\end{equation}
The set of conjugacy classes $\CC(S_n)$ is in bijection with the set of partitions of $n$, which encode the cycle decomposition of permutations. All generators fall into the class $\OO_{2,1,\ldots,1}$, with $k(\OO_{2,1,\ldots,1})=2$. The central element $t_{\OO_{2,1,\ldots,1}}$ can be written as $e_{\sigma_i}^2$ for any $i$. Thus $\#\CC(S_n)=P(n)$ (the partition number of $n$), $\CC_{\ig}(S_n)$ is empty, $\#\CC_{\fg}(S_n)=1$, and $(S_n)_{\ff} = \Z_2$. 
\end{xmp}

\begin{xmp}\label{Ex:BnPresentation}
For the closely related braid groups $B_n$, a \cc-type presentation is obtained by removing the power relations $\sigma_i^2=1$ from the above presentation for $S_n$. The set $\CC(B_n)$ is now infinite, $\CC_{\fg}(B_n)$ is empty, and $\#\CC_{\ig}(B_n)=1$.  
\end{xmp}

\begin{proof}[Proof of Theorem~\ref{T:Inj}]
Take an element $g\in A(G)$ satisfying $\pi(g)=1_G$ and $\varepsilon_{\OO}(g)=0$ for all $\OO \in \CC(G)$. We need to explain why $g$ is trivial in $A(G)$. Our argument is based on rewriting.

Fix a representative $a_{\OO} \in \OO$ for any $\OO \in \CC(G)$; if $\OO \in \CC_{\gg}(G)$, we require $a_{\OO} \in \Gamma$. Any other $b$ from the same conjugacy class can be written as $b=a_{\OO} \op a$ for some $a \in G$. Express $a$ in terms of the generators of $\Pi$: $a=a_1^{\alpha_1}\cdots a_s^{\alpha_s}$, with $a_i \in \Gamma,\ \alpha_i = \pm 1$. Then 
\[b=a_{\OO} \op(a_1^{\alpha_1}\cdots a_s^{\alpha_s}) = (\cdots(a_{\OO} \op^{\alpha_1} a_1) \cdots \op^{\alpha_s} a_s,\]
where $\op^{+1}=\op$, and the operation $\op^{-1}$ is defined by $a \op b = c \Leftrightarrow c \op^{-1} b = a$. Then the defining relation~\eqref{Eq:StrGroups} of the structure group yields
\begin{equation}\label{E:ToRepresentatives}
e_b=(\cdots(e_{a_{\OO}} \op^{\alpha_1} e_{a_1}) \cdots \op^{\alpha_s} e_{a_s}.
\end{equation}

Now, express our $g\in A(G)$ in terms of the generators $e_a$ of $A(G)$.  Using \eqref{Eq:StrGroups}, these generators can be rearranged according to their conjugacy classes, which we order in an arbitrary way (the indices $a$ in $e_a$ might change along the way, but not their conjugacy class). Apply procedure \eqref{E:ToRepresentatives} to all the $e_a$ with $a \in \sqcup_{\OO \in \CC_{\ng}(G)} \OO$. This presents $g$ in terms of the $e_{a_{\OO}},\ \OO \in \CC_{\ng}(G)$, and the $e_a,\ a \in \sqcup_{\OO \in \CC_{\gg}(G)} \OO$. Using \eqref{Eq:StrGroups}, the latter can be pushed through the $e_{a_{\OO}}$ to the right of the expression, receiving $\op$-actions on the way. Now, apply procedure \eqref{E:ToRepresentatives} once again, this time to all the $e_a$ with $a \in \sqcup_{\OO \in \CC_{\gg}(G)} \OO$. The result can be re-organised as 
\begin{equation}\label{E:ToGenerators}
g=e_{a_{\OO_1}}^{d_1} e_{a_{\OO_2}}^{d_2} \cdots e_{b_1}^{\beta_1}\cdots e_{b_t}^{\beta_t}, \hspace*{1cm} \OO_i \in \CC_{\ng}(G),\ d_i \in \Z, \ b_j \in \Gamma, \ \beta_j = \pm 1,
\end{equation}
where only a finite number of the $d_i$'s are non-zero. 

On the one hand, for any $\OO_i \in \CC_{\ng}(G)$, one has $\varepsilon_{\OO_i}(g)=0$. On the other hand, \eqref{E:ToGenerators} yields $\varepsilon_{\OO_i}(g)=d_i$. Hence all the $d_i$ vanish. Further, since $1_G=\pi(g)= b_1^{\beta_1}\cdots b_t^{\beta_t}$, the word $b_1^{\beta_1}\cdots b_t^{\beta_t}$ can be made trivial in $G$ using only relations from the presentation $\Pi$. But this procedure can be imitated in $A(G)$ for the element $g$, since a relation $a \op b = c$ from $\Pi$ has an analogue $e_a \op e_b = a_{a \op b} = e_c$ in $A(G)$, and a relation $a^k=1$  has an analogue $e_a^k = t_{\OO}$, where $a \in \OO \in \CC_{\fg}(G)$, and $t_{\OO}$ is defined in~\eqref{E:T}. As the $t_{\OO}$'s are central, one can keep them all at the beginning of the word, and the rewriting procedure eventually leaves us with an expression
\[g=t_{\OO_1}^{p_1} t_{\OO_2}^{p_2} \cdots, \hspace*{2cm} \OO_i \in \CC_{\fg}(G),\ p_i \in \Z.\]
As usual, only a finite number of the $p_i$'s are non-zero. Now, for any $\OO_i$, which is from $\CC_{\fg}(G)$ this time, one has $0=\varepsilon_{\OO_i}(g)=p_i k(\OO_i)$. Hence all the $p_i$'s vanish, and $g=1_{A(G)}$ as desired.
\end{proof}

\begin{rmr}\label{R:BetterInclusion}
We actually proved a stronger (but less elegant) group inclusion
\[A(G) \ \hookrightarrow \ G \times \Z^{\CC_{\fg}(G)} \times \Z^{\CC_{\ng}(G)},\]
obtained from the one in the theorem by ignoring the $\Z^{\CC_{\ig}(G)}$ part of $\Z^{\CC(G)}$.  
\end{rmr}

The theorem allows us to reduce many computations in $A(G)$ to $G$, especially those related to commutativity. The most straightforward example is

\begin{crl}\label{C:Comm}
Let $G$ be a \cb-group. Then elements $f$ and $g$ commute in $A(G)$ if and only if $\pi(f)$ and $\pi(g)$ commute in $G$.
\end{crl}

\begin{proof}
Denoting the commutator as 
\[ [f,g]=fgf^{-1}g^{-1},\] 
one has 
\[(\pi \times \Ab)([f,g]) \ = \ ([\pi(f),\pi(g)],0) \in G \times \Z^{\CC(G)},\]
so, by the injectivity of $\pi \times \Ab$, we get $ [f,g]=1_{A(G)} \ \Leftrightarrow \ [\pi(f),\pi(g)]=1_G$.
\end{proof}

Similarly, one proves
\begin{crl}\label{C:Center}
Let $G$ be a \cb-group. Then the center of $A(G)$ can be computed as 
\[Z(A(G)) \ = \ \pi^{-1}(Z(G)).\]
In particular, for a centerless group, one has $Z(A(G)) = \Ker(\pi)$.
\end{crl}

\begin{crl}\label{C:Derived}
Let $G$ be a \cb-group. The map $\pi$ restricts to an isomorphism of the derived subgroups:
\[A(G)' \ \cong\ G'.\]
\end{crl}

\begin{proof}
Recall that the derived group is generated by the commutators. Since $\pi([f,g]) = [\pi(f),\pi(g)]$, the map $\pi$ restricts to a map $A(G)' \to G'$. Since $\Ab([f,g]) = 0$, the theorem guarantees that this restriction is injective.  Since $\pi([e_a,e_b]) = [a,b]$ for all $a,b \in G$, it is also surjective.
\end{proof}

\begin{crl}\label{C:Tor}
Let $G$ be a \cb-group. The torsion of $A(G)$ is
\[\Tor(A(G)) \ = \ \Ker(\Ab) \cap \pi^{-1}(\Tor(G)),\] 
which is isomorphic to a subgroup of $\Tor(G)$. 
\end{crl}

\begin{proof}
By the injectivity of $\pi \times \Ab$, $g$ is a torsion element in $A(G)$ if and only if $\pi(g)$ is a torsion element in $G$, and $\Ab(g)$ is a torsion element in a free abelian group, which means $\Ab(g)=0$. The injectivity of $\pi \times \Ab$ is then used once more to show the injectivity of the restriction of $\pi$ to $\Ker(\Ab) \cap \pi^{-1}(\Tor(G)) \to \Tor(G)$. 
\end{proof}

This is particularly useful in the two extreme cases:

\begin{enumerate}
\item If $G$ is finite (e.g., $G=S_n$), then $\Tor(G)=G$, hence
\[\Tor(A(G)) \ =\ \Ker(\Ab) \ =\ A(G)' \ \cong\ G'.\] 
\item If, on the contrary, $G$ has no torsion (e.g., $G=B_n$), then neither has $A(G)$:
\begin{align*}
\Tor(A(G)) \ &= \ \Ker(\Ab) \cap \pi^{-1}(\Tor(G)) \ = \ \Ker(\Ab) \cap \Ker(\pi) \\
& = \ \Ker(\pi \times \Ab) \ = \ \{1_{A(G)}\}.
\end{align*}
\end{enumerate}

\begin{xmp}
Our favourite \cb-group $S_n$ is finite and centerless for $n>2$. The above corollaries then yield:
\[\Tor(A(S_n)) \ = \ \Ker(\Ab) \ =\ A(S_n)'\ \cong\ S_n' \ = \ A_n,\]
\[Z(A(S_n)) = \Ker(\pi).\]
Here $A_n$ is the alternating group.
\end{xmp}

\begin{xmp}
Our favourite \cc-group $B_n$ has no torsion, hence neither has $A(B_n)$. The center of $B_n$ is cyclic, $Z(B_n)=\langle \Delta_n^2 \rangle$ for $n>2$, thus $Z(A(B_n))$ is generated by $\Ker(\pi)$ and $e_{\Delta_n^2} \in \pi^{-1}(\Delta_n^2 )$. This yields an isomorphism $Z(A(B_n)) \cong \Ker(\pi) \times \Z$.
\end{xmp}

\section{Structure groups as pullbacks}\label{Sec:Pullback}

Let us now describe the image of the group inclusion $\pi \times \Ab \colon  A(G) \hookrightarrow G \times \Z^{\CC(G)}$ for a \cb-group $G$. 

\begin{lmm}\label{L:opi} 
There exists a unique group morphism $\opi$ completing the following commutative square:
\[
\begin{tikzcd}
A(G) \arrow[two heads]{r}{\Ab} \arrow[two heads,swap]{d}{\pi} & \Z^{\CC(G)} \arrow[two heads, dashed]{d}{\opi} \\
G \arrow[two heads]{r}{\Ab} & \Z^{\CC_{\ig}(G)} \bigoplus \Gf = \Ab(G)
\end{tikzcd}
\]
(where we use the computation~\eqref{E:AbG}).
\end{lmm}

\begin{proof}
Given a generator $c_{\OO}$ of $\Z^{\CC(G)}$, take any of its $Ab$-preimages $\widehat{c_{\OO}} \in A(G)$, and declare $\opi(c_{\OO})$ to be $\Ab(\pi(\widehat{c_{\OO}}))$. Different choices of the lift $\widehat{c_{\OO}}$ differ by commutator factors $ghg^{-1}h^{-1}$, with $g,h \in A(G)$. But these are sent by $\pi$ to commutators in $G$, which are killed by $\Ab$:
\[ghg^{-1}h^{-1} \overset{\pi}{\mapsto} \pi(g)\pi(h)\pi(g)^{-1}\pi(h)^{-1} \overset{\Ab}{\mapsto} 0.\]
 Thus our $\opi(c_{\OO})$ does not depend on the choice of $\widehat{c_{\OO}}$. Since the target group is abelian, we get a well defined group morphism $\opi$. It is surjective since the $c_{\OO}$'s with $\OO \in \CC_{\gg}(G)$ are sent to the corresponding generators of $\Ab(G)$. The square obtained is commutative by the construction of $\opi$, and this completion of the square is unique since the upper map $\Ab$ is surjective.
\end{proof} 

\begin{thm}\label{T:Pullback}   
For a \cb-group $G$, the structure group $A(G)$ is the pullback
\[ A(G) \ \cong \  G \underset{\Ab,\Ab(G),\opi}{\times} \Z^{\CC(G)}.\]
\end{thm} 
The above pullback will be denoted as $G \underset{\Ab(G)}{\times} \Z^{\CC(G)}$ for brevity.

\newpage 
\begin{xmp}
The pullback looks far less scary in our favourite example $G=S_n$. Here $\Z^{\CC(S_n)} \cong \Z^{P(n)}$, $\Ab(S_n) \cong \Z_2$, and the maps $\Ab \colon S_n \to \Z_2$ (resp., $\opi \colon \Z^{P(n)} \to \Z_2$) are the signature maps: they send even permutations (resp., generators corresponding to even partitions) to $0$ and the remaining ones to $1$. Thus $A(S_n)$ is really close to being a direct product:
\[ A(S_n) \ \cong \  S_n \underset{\Z_2}{\times} \Z^{P(n)}.\]
\end{xmp}

\begin{xmp}
In the case of braid groups, the group $\Z^{\CC(B_n)}$ is infinitely generated; $\Ab(B_n) \cong \Z$; the map $\Ab \colon B_n \to \Z$ computes the length $\lg$ of a word written in the standard (equivalently, extended standard, in the sense of \eqref{E:PresentationSn}) generators; and the map $\opi \colon \Z^{\CC(B_n)}\to \Z$ sends the generator $c_{\OO}$ to $\operatorname{lg}(a)$ for any $a \in \OO \subset B_n$.
Hence
\[ A(B_n) \ \cong \  B_n \underset{\Z}{\times} \Z^{\CC(B_n)}.\]
\end{xmp}

\begin{proof}[Proof of Theorem~\ref{T:Pullback}] 
Recall the definition of the pullback:
\[G \underset{\Ab(G)}{\times} \Z^{\CC(G)} = \{\ (a,x)\ | \ a \in G,  \ x \in \Z^{\CC(G)},\ \Ab(a)=\opi(x). \ \}\]
Theorem~\ref{T:Inj} together with Lemma~\ref{L:opi} guarantee that the injection $\pi \times \Ab$ actually takes values in $G \underset{\Ab(G)}{\times} \Z^{\CC(G)}$. We need to show the surjectivity, that is, given $a \in G$ and $x \in \Z^{\CC(G)}$ with $\Ab(a)=\opi(x)$, we are looking for a $g \in A(G)$ satisfying $\pi(g)=a$ and $\Ab(g)=x$.

Assume first that $x \in \Z^{\CC_{\gg}(G)}$, that is, its components indexed by $\OO \in \CC_{\ng}(G)$ are trivial. Write $a$ as $a_1^{\alpha_1}\cdots a_s^{\alpha_s}$, with $a_i \in \Gamma,\ \alpha_i = \pm 1$. Put $g'= e_{a_1}^{\alpha_1}\cdots e_{a_s}^{\alpha_s}$. It is almost what we need: $\pi(g')=a$, $p_{\OO}(\Ab(g'))=0$ for $\OO \in \CC_{\ng}(G)$, and 
\begin{equation}\label{E:PullbackComponents}
p_{\OO}(\Ab(g'))=p_{\OO}(\Ab(a))=p_{\OO}(\opi(x)) = p_{\OO}(x)
\end{equation}
 for $\OO \in \CC_{\ig}(G)$, where $p_{\OO}$ projects $\Z^{\CC(G)}$ or $\Z^{\CC_{\ig}(G)} \bigoplus \Gf$ onto the component indexed by~$\OO$. The case of conjugacy classes from $\CC_{\fg}(G)$ is more subtle, since in this case $p_{\OO}(\Ab(a))=p_{\OO}(\opi(x)) \in \Z_{k(\OO)}$, whereas $p_{\OO}(\Ab(g'))$ and $p_{\OO}(x)$ live in $\Z$. Equalities \eqref{E:PullbackComponents} then hold modulo $k(\OO)$ only. But then $g'$ can be replaced with 
 \[g=g' t_{\OO_1}^{p_1} t_{\OO_2}^{p_2} \cdots, \hspace*{2cm} \OO_i \in \CC_{\fg}(G),\ p_i \in \Z,\]
 with the exponents $p_i$ chosen to get the equalities $p_{\OO}(\Ab(g)) = p_{\OO}(x)$ in $\Z$ for all $\OO \in \CC_{\fg}(G)$. Since $p_{\OO}(\Ab(g)) =p_{\OO}(\Ab(g'))=p_{\OO}(x)$ for other $\OO$'s, we get $\Ab(g)=x$. Finally, recalling that $\pi(t_{\OO}) = \pi(e_{a_{\OO}}^{k(\OO)}) =a_{\OO}^{k(\OO)} = 1_G$, we get $\pi(g)=\pi(g')=a$, and we are done.

In the general case, consider $x+y \in \Z^{\CC(G)}$, with $x \in \Z^{\CC_{\gg}(G)}$ and $y \in \Z^{\CC_{\ng}(G)}$, and $b \in G$ such that $\Ab(b)=\opi(x+y)$. Since the map $\Ab$ is surjective, there is some $h \in A(G)$ with $\Ab(h)=y$. One can decompose $b$ as $(b \pi(h)^{-1}) \pi(h)$. Since 
\[\Ab(b \pi(h)^{-1}) = \Ab(b) - \Ab(\pi(h))=\opi(x+y)- \opi(\Ab(h))=\opi(x)+\opi(y)- \opi(y) =\opi(x),\]
the above argument yields a $g \in A(G)$ satisfying $\pi(g)=b \pi(h)^{-1}$ and $\Ab(g)=x$. But then $gh \in A(G)$ satisfies 
\[\pi(gh)=\pi(g)\pi(h) = (b \pi(h)^{-1}) \pi(h) = b\]
 and $\Ab(gh)=\Ab(g) + \Ab(h) =x + y$, as requested. 
\end{proof}

\begin{rmr}
The lemma and the theorem above are in fact a particular case of Goursat's lemma, valid for any subdirect product of groups. 
\end{rmr}

\section{Structure groups as direct products}\label{Sec:Dehn}

In the proof of Theorem~\ref{T:Inj}, we have seen that generators $e_a$ of $A(G)$ with $a \in \CC_{\ng}(G)$ and those with $a \in \CC_{\gg}(G)$ play very different roles. We will now show how to split $A(G)$ into two parts, a trivial one corresponding to $\CC_{\ng}(G)$, and an interesting one  corresponding to $\CC_{\gg}(G)$.

To do this, we will use a \cb-type presentation $\Pi$ of $G$, with generator set $\Gamma$, to produce many elements in $\Ker(\pi) \subseteq Z(A(G))$. Concretely, for any $\OO \in \CC_{\ng}(G)$, take an $a \in \OO$, write it as $a=a_1^{\alpha_1}\cdots a_s^{\alpha_s}$, with $a_i \in \Gamma,\ \alpha_i = \pm 1$. Put 
\begin{equation}\label{E:Tng}
t_{\OO} := e_a e_{a_s}^{-\alpha_s}\cdot e_{a_1}^{-\alpha_1} \in \Ker(\pi), \hspace*{1cm} \OO \in \CC_{\ng}(G).
\end{equation}
Recall also the elements $t_{\OO} \in \Ker(\pi)$ defined in \eqref{E:T} for all $\OO \in \CC_{\fg}(G)$.

\begin{prp}\label{P:ManyCentralElements}
The above elements $t_{\OO}$, with $\OO \in \CC_{\ng}(G) \sqcup \CC_{\fg}(G)$, are free abelian generators of a subgroup of $\Ker(\pi)$ isomorphic to $\Z^{\CC_{\ng}(G)} \times \Z^{\CC_{\fg}(G)}$.
\end{prp}

We will later prove that this recovers the whole $\Ker(\pi)$.

\begin{proof}
The freeness can be seen in the abelianisation. Indeed, for $\OO \in \CC_{\ng}(G)$,
\[\varepsilon_{\OO}(t_{\OO'}) = \begin{cases} 1 &\text{ if } \OO'=\OO,\\
0 &\text{ otherwise.}\end{cases}\]
Further for $\OO, \OO' \in \CC_{\fg}(G)$,
\[\varepsilon_{\OO}(t_{\OO'}) = \begin{cases} k_{\OO} &\text{ if } \OO'=\OO,\\
0 &\text{ otherwise.}\end{cases} \qedhere\]
\end{proof}

This is the trivial part of $A(G)$. Let us now describe the interesting part.

\begin{dfn} We call the \emph{Dehn quandle} of $(G,\Pi)$ the union $\Gamma^G$ of the $G$-conjugacy classes of the generator set $\Gamma$ of $\Pi$. Equivalently, it is the subquandle of $\Conj(G)$ generated by $\Gamma$.
\end{dfn}

The structure group of this quandle will be denoted by
\[G^{\Deh}\ :=\ \As(\Gamma^G) \ = \ \langle e_a, \, a \in \Gamma^G \,|\, e_a e_b = e_b e_{a\op b},\ a, b \in \Gamma^G \rangle.\]

The arguments from the proofs of Theorems \ref{T:Inj} and \ref{T:Pullback} can be applied verbatim to $G^{\Deh}$, simplified significantly since one ne longer needs to take the conjugacy classes from $\CC_{\ng}(G)$ into account. This yields:

\begin{thm}\label{T:InjForDehn}
For a \cb-group $G$, the map
\[\pi \times \Abf \colon G^{\Deh} \to G \times \Z^{\CC_{\fg}(G)}\]
is injective.
\end{thm} 
Here we abusively keep the notation $\pi$ for the map $G^{\Deh} \to G$, $e_a \mapsto a$; and $\Abf$ is the composition of the abelianisation map $G^{\Deh} \to \Z^{\CC_{\gg}(G)} = \Z^{\CC_{\fg}(G)} \oplus \Z^{\CC_{\ig}(G)}$ with the projection onto its part $\Z^{\CC_{\fg}(G)}$ corresponding to finite-order generators. 

\begin{prp}\label{P:PullbackForDehn}
For a \cb-group $G$, the structure group $G^{\Deh}$ is the pullback
\[ G^{\Deh}\ \cong \  G \underset{\Abf,\Gf,\opi}{\times} \Z^{\CC_{\fg}(G)},\] 
where :
\begin{itemize}
\item $\opi \colon \Z^{\CC_{\fg}(G)} \twoheadrightarrow \Gf = \bigoplus_{\OO \in \CC_{\fg}(G)} \Z_{k(\OO)}$ is the obvious projection of each copy of $\Z$ in $\Z^{\CC_{\fg}(G)}$ indexed by $\OO \in \CC_{\fg}(G)$ onto $\Z_{k(\OO)}$;
\item $\Abf \colon G \to \Gf$ is the composition of the abelianisation map $G \twoheadrightarrow \Z^{\CC_{\ig}(G)} \oplus \Gf$ with the projection onto $\Gf$.
\end{itemize}
\end{prp} 

In other words, one has a pullback
\[
\begin{tikzcd}
G^{\Deh} \arrow[two heads]{r}{\Abf} \arrow[two heads,swap]{d}{\pi} & \Z^{\CC_{\fg}(G)} \arrow[two heads]{d}{\opi} \\
G \arrow[two heads]{r}{\Abf} & \Gf.
\end{tikzcd}
\]

\begin{xmp}
For the presentation \eqref{E:PresentationSn} of $S_n$, $\Gamma^{S_n} \subset S_n$ is the subset $\Transp_n$ of all transpositions, endowed with the conjugation operation. Our pullback becomes
\[
\begin{tikzcd}
S_n^{\Deh} =\As(\Transp_n) \arrow[two heads]{r}{\varepsilon} \arrow[two heads,swap]{d}{\pi} & \Z \arrow[two heads]{d} \\
S_n \arrow[two heads]{r}{\operatorname{sign}} & \Z_2.
\end{tikzcd}
\]
The horizontal maps are abelianisations: the degree map \eqref{E:Degree}, and the signature map for permutations. The rightmost map is the usual quotient $\Z \twoheadrightarrow \Z_2$.
\end{xmp}

The group $G^{\Deh}$ can be described in a more accessible way:

\begin{thm}\label{T:DehnGenRel}
Fix a \cb-type presentation $\Pi$ of a group $G$, with generator set $\Gamma$. A presentation of the group $G^{\Deh}$ can then be obtained from $\Pi$ as follows:
\begin{itemize}
\item \underline{generators}:\ $e_a, \ a \in \Gamma$;
\item \underline{relations} of two types:
\begin{enumerate}
\item $e_a \op e_b = e_c$ \hspace*{.3cm} for each conjugation relation $a \op b = c$ in $\Pi$;
\item $e_a^k e_b = e_b e_a^k$ \hspace*{.3cm} for all $b \in \Gamma$, and each power relation $a^k=1$ in $\Pi$.
\end{enumerate}
\end{itemize}
\end{thm} 
Thus $G^{\Deh}$ is a relaxed version of $G$: the powers $e_a^k$ are required to be central rather than trivial. We will call it the \emph{Dehn-type lift} of $G$, and denote the presentation from the statement of the theorem as $\Pi^{\Deh}$. 

\begin{proof}
Relations $e_a \op e_b = e_c$ hold in $G^{\Deh}$ by its definition, and the powers $e_a^k$ are central due to Lemma \ref{L:CentralPowers}. The $e_a$ with $a \in \Gamma$ generate $G^{\Deh}$ using the rewriting mechanism from \eqref{E:ToRepresentatives}. It remains to check that if a word
\[w=e_{b_1}^{\beta_1}\cdots e_{b_s}^{\beta_s}, \hspace*{1cm} b_i \in \Gamma,\ \beta_i = \pm 1,\]
represents a trivial element $g \in G^{\Deh}$, then it can be made trivial using relations from the presentation $\Pi^{\Deh}$ only. We will use the same rewriting strategy as in the proof of Theorem \ref{T:Inj}. Since $\pi(g)=1_G$, one can imitate a rewriting sequence transforming $b_1^{\beta_1}\cdots b_s^{\beta_s}$ into $1_G$, to transform $w$ into a finite product of terms $(e_{a_i}^{k_i})^{d_i}$, at most one for each power relation $e_{a_i}^{k_i}=1$ in $\Pi$, using the presentation $\Pi^{\Deh}$ only. But since $\Abf(g)=0 \in \Z^{\CC_{\fg}(G)}$, all the exponents $p_i$ must be trivial.
\end{proof}

One can relax the power relations $a^k=1$ even further, by simply forgetting them. This yields an \emph{Artin-type lift} of $G$:

\begin{dfn}
Given a group $G$ with a \cb-type presentation $\Pi$, the group $G^{\Art}$ is defined by keeping all the generators and only conjugation relations from $\Pi$.
\end{dfn} 

Theorem \ref{T:DehnGenRel} and Lemma \ref{L:CentralPowers} now immediately imply Theorem \ref{T:DehnAsIntermediate}. Theorem \ref{T:CGroupIFF} requires slightly more work:

\begin{proof}[Proof of Theorem \ref{T:CGroupIFF}]
For a \cc-group $G$, the three groups above have the same presentation, hence
\[ G \ \cong \  G^{\Deh} \ = \ \As(\Gamma^G),\]
which proves the first part of the theorem. 

Further, for a finite \cb-group $G$, the finite quotient $\oAs(\Gamma^G)$ is the quotient of $\As(\Gamma^G)$ by the powers $e_a^k$ of all generators chosen to have trivial action on $G$. But this is equivalent to $k$ being the minimal positive integer rendering $a^k$ central in $G$. If $G$ has trivial center, this translates as $a^k=1$, which means that $\Pi$ contains a conjugate of the power relation $a^k=1$ (the minimality follows by passing to the abelianisation). By Theorem \ref{T:DehnAsIntermediate}, the group $G$ can be regarded as a quotient of $G^{\Deh} = \As(\Gamma^G)$ by the same elements. Therefore $G \cong \oAs(\Gamma^G)$. Hence the second part of the theorem.
\end{proof}

\begin{xmp}
For $S_n$, the Artin-type lift gives the presentation for the braid group $B_n$ recalled in Example \ref{Ex:BnPresentation}. The group $S_n^{\Deh}$ is then obtained from $B_n$ by imposing the centrality of the square of a generator, say, $\sigma_1^2$. This is summarised as follows:

\vspace*{-.2cm}
\begin{align*}
S_n^{\Art} \ \cong \ \As(&\Gamma^{B_n})  \ \cong \ B_n \\
&\twoheadrightarrowdown\\
S_n^{\Deh} \ =\ \As(\Transp_n) \ &\cong \ \raisebox{.2em}{$B_n \times \langle t \rangle$}\left/\raisebox{-.2em}{$\sigma_1^2=t$}\right.  \\
&\twoheadrightarrowdown \\ 
 S_n \ \cong \ \oAs(\Transp_n&) \cong \ \raisebox{.2em}{$B_n$}\left/\raisebox{-.2em}{$\sigma_1^2=1$}\right.
\end{align*}
Here $\Gamma^{B_n}$ is the set of all conjugates of the generator $\sigma_1$ in $B_n$.
\end{xmp}

\begin{rmr}
Theorem \ref{T:CGroupIFF} implies in particular that the image of the structure group functor is the category of \cc-groups. 
\end{rmr}

\medskip
We are now ready to prove the promised splitting of $A(G)$ into two parts.
\begin{thm}\label{T:DirectProductDetailed}
For any \cb-group $G$, one has a group isomorphism
\begin{align*}
\phi \colon G^{\Deh} \times \Z^{\CC_{\ng}(G)} \ &\overset{\sim}{\to} \ A(G),\\
e_a \ &\mapsto \ e_a, \hspace*{1cm} a \in \Gamma^G,\\
c_{\OO} \ &\mapsto  \ t_{\OO}, \hspace*{1cm} \OO \in \CC_{\ng}(G),
\end{align*}
where $c_{\OO}$ is the generator of $\Z^{\CC_{\ng}(G)}$ corresponding to $\OO$, and $t_{\OO}$ is defined in \eqref{E:Tng}.
\end{thm} 

This yields Theorem~\ref{T:DirectProduct} from the Introduction.

\begin{proof}
The map $\phi$ is well defined because:
\begin{itemize}
\item the defining relations for the generators $e_a$ in $G^{\Deh}$ are the same as for the corresponding generators in $A(G)$;
\item the $t_{\OO}$'s all lie in $\Ker(\pi)$, and are thus central in $A(G)$%
\end{itemize}

It is surjective since its image contains all the generators of $A(G)$:
\begin{itemize}
\item the $e_a$'s with $a \in \Gamma$ as the images of the corresponding $e_a$'s from $G^{\Deh}$;
\item $e_a$ for a representative $a$ of each $\OO \in \CC_{\ng}$, constructed from $\phi(c_{\OO}) = t_{\OO}$ and from the $e_{a_i}$'s with $a_i \in \Gamma$;
\item all other $e_a$'s as conjugates of the ones describes above, via \eqref{E:ToRepresentatives}.
\end{itemize}

To show the injectivity, pick some $(g,x)\in \Ker(\phi)$. Applying to $\phi(g,x)$ the $\varepsilon_{\OO}$'s with $\OO \in \CC_{\ng}(G)$, one sees that the corresponding components of $x \in \Z^{\CC_{\ng}(G)}$ must be trivial, hence $x=0$. But then, composing $\phi$ with $\pi \times \Ab$, we see that $(\pi \times \Abf)(g)$ is trivial. By Theorem \ref{T:InjForDehn}, this yields the triviality of $g$.
\end{proof}

\medskip
An immediate consequence of the theorem is the (non-obvious!) group inclusion
\begin{align}
\xi \colon \As(\Gamma^G) = G^{\Deh} & \hookrightarrow A(G),\label{E:GDinAG}\\
e_a &\mapsto e_a.\notag
\end{align}

Since the factor $\Z^{\CC_{\ng}(G)}$ is abelian, this implies

\begin{crl}\label{C:Derived2}
Let $G$ be a \cb-group. The map $\xi$ restricts to an isomorphism of the derived subgroups:
\[(G^{\Deh})' \ \cong\ A(G)' \ \cong\ G'.\]
\end{crl}
We used Corollary \ref{C:Derived} for the second isomorphism. Note that the groups $G^{\Art}$ and $G$ have different derived groups in general: for instance, $(S_n)'=A_n$ is finite, while $S_n^{\Art} \cong B_n$ has no torsion.

This property is particularly useful when all generators in $\Pi$ are from the same conjugacy class:

\begin{prp}\label{P:SemiDirect}
Let $\Pi$ be a \cb-presentation of a group $G$ such that all its generators fall into the same conjugacy class. Then $G^{\Deh}$ decomposes as a semi-direct product
\[G^{\Deh}   \ \cong \ G' \rtimes \Z,\]
where $1 \in \Z$ acts on $G'$ by conjugation by a chosen generator $a$ of $\Pi$.
\end{prp}

\begin{proof}
In the short exact sequence 
\[ 0 \ \to \ \Ker(\Ab) \ \to \ G^{\Deh} \ \overset{\Ab}{\to} \ \Z \ \to \  0,\]
one can replace $\Ker(\Ab)$ by $(G^{\Deh})'$, and then use the section
\begin{align*}
\Z &\to G^{\Deh},\\
k &\mapsto e_a^k.
\end{align*}
This yields $G^{\Deh} \cong (G^{\Deh})' \rtimes \Z$, with $k \in \Z$ acting on $(G^{\Deh})'$  by conjugation by $e_a^k$. Under the isomorphism $(G^{\Deh})' \cong G'$, this becomes conjugation by $a^k$. 
\end{proof}

\begin{xmp}
For $S_n$, we recover the decomposition
\[ S_n^{\Deh} \ \cong \ A_n \rtimes \Z,\]
where $1 \in \Z$ acts on $A_n$ by conjugation by, say, the generator $\sigma_1$.
\end{xmp}

\medskip
Finally, recall that for a \cc-group $G$, the group $G^{\Deh}$ is simply $G$ itself. This gives
\begin{crl}\label{C:DirectProductCGroup}
For any \cc-group $G$, the structure group $A(G)$ splits as 
\[A(G) \ \cong \ G \times \Z^{\CC_{\ng}(G)}.\]
\end{crl} 
In other words, $\As(\Conj((G))$ is only a trivial extension of $G$.

\section{Structure group $A(G)$ as an extension of $G$}\label{Sec:Ext}

Let us now focus on the map $\pi$ connecting the structure group $A(G)$ to the original group $G$, which we suppose to be arbitrary (not necessarily of \cb-type) for the moment. The difference between the two groups is measured by the central (hence abelian) subgroup $K(G):=\Ker(\pi)$ of $A(G)$, which is hard to compute in general. 

The map $\pi$ has an obvious \emph{set-theoretic section}
\begin{align*}
s \colon G &\to A(G),\\
 a &\mapsto e_a.
\end{align*}  
It is not a group morphism: $e_ae_b$ and $e_{ab}$ always have different degrees. It is classical to study the \emph{``morphicity defect''} of such a section, which is a map
\begin{align*}
\varphi \colon G \times G &\to K(G),\\
 (a,b) &\mapsto e_{ab}^{-1}e_ae_b.
\end{align*}
Such maps govern extensions of $G$ by $K(G)$. This is written as
\[A(G) \ \cong \ G \underset{\varphi}{\times} K(G).\]

\begin{prp}\label{P:2cocycle}
The above map $\varphi$ is a symmetric equivariant group $2$-cocycle. That is, it satisfies the following properties for all $a,b,c \in G$:
\begin{align}
& \varphi(b,c)-\varphi(ab,c)+\varphi(a,bc)-\varphi(a,b)=0,\label{E:2cocycle}\\
& \varphi(a,1)=e_1,\label{E:Norm}\\
& \varphi(a \op c,b \op c)=\varphi(a,b),\label{E:Equivar}\\
& \varphi(a,b) = \varphi(b,a). \label{E:Sym}
\end{align}
\end{prp}

Thus in principle the knowledge of the second group cohomology of $G$ could shed light on the structure group $A(G)$.

\begin{proof}
The proof of the $2$-cocycle condition \eqref{E:2cocycle} is classical:
\begin{align*}
\varphi(a,bc) + \varphi(b,c) &= e_{a(bc)}^{-1}e_ae_{bc} e_{bc}^{-1}e_be_c = e_{abc}^{-1}e_ae_be_c,\\
\varphi(ab,c)+\varphi(a,b)  &= e_{(ab)c}^{-1}e_{ab}e_{c} e_{ab}^{-1}e_ae_b = e_{abc}^{-1}e_{ab}e_{ab}^{-1}e_ae_be_{c}  = e_{abc}^{-1}e_ae_be_c,
\end{align*}
where we used the centrality of $e_{ab}^{-1}e_ae_b=\varphi(a,b)$. Note that we alternate between the multiplicative notation in the whole group $A(G)$, and the additive notation in the abelian subgroup $K(G)$.

The normalisation condition \eqref{E:Norm} is straightforward.

The equivariance condition \eqref{E:Equivar} follows from 
\begin{align*}
\varphi(a \op c,b \op c) &= e_{(a \op c)(b \op c)}^{-1}e_{a  \op c} e_{b \op c} = e_{(ab) \op c}^{-1}e_{a  \op c} e_{b \op c} = (e_{ab}^{-1} \op e_c)(e_a \op e_c)(e_b \op e_c)\\
&= (e_{ab}^{-1}e_ae_b) \op e_c = e_{ab}^{-1}e_ae_b,
\end{align*} 
since the central element $e_{ab}^{-1}e_ae_b$ commutes with $e_c$.

The symmetry condition \eqref{E:Sym} follows from 
\begin{align*}
\varphi(b,a) &= e_{ba}^{-1}e_be_a = e_{a (b \op a)}^{-1}e_a(e_b \op e_a) = e_{(a \op a) (b \op a)}^{-1} e_{a \op a} e_{b \op a} \\
& = \varphi(a \op a,b \op a) = \varphi(a,b).
\end{align*} 
We used the equivariance condition \eqref{E:Equivar}.
\end{proof}

It is remarkable that, in some sense, the $2$-cocycle $\varphi$ cannot be too trivial:

\begin{prp}\label{P:2cocycleImage}
The image of $\varphi$ generates the whole abelian group $K(G)$. 
\end{prp}

\begin{proof}
Take a word $w$ in the letters $e_a,\ a \in G$, representing some $g \in K(G)$. Using the relation
\[\varphi(a,b) e_{ab} = e_a e_b,\]
one can replace two neighbouring letters by one, at the cost of introducing a central factor $\varphi(a,b)^{\pm 1}$. Repeating this procedure, one gets one letter $e_d^{\pm 1}$ times some $\varphi$ factors. Since $\pi(g)=1$ and $\pi(\varphi(a,b))=1$, one gets $d^{\pm 1} = \pi(e_d^{\pm 1})=1$, hence $d=1$. But $e_1=\varphi(1,1)$. In total, we have presented $g$ as a product of the values of $\varphi$ and their inverses.
\end{proof}

\medskip
Let us now return to \cb-groups. Our understanding of the structure groups $A(G)$ in this case is sufficient to determine $K(G)$ completely.

\begin{prp}\label{P:Kernel}
A \cb-group $G$ fits into the exact sequence
\[0 \ \to \ \Z^{\CC_{\ng}(G)} \times \Z^{\CC_{\fg}(G)} \ \overset{\iota}{\to} \ A(G) \ \overset{\pi}{\to} \ G \ \to \ 0,\] 
where $\iota$ is the map constructed in Proposition \ref{P:ManyCentralElements}.
\end{prp}

Hence $K(G)$ is the free abelian group $\Z^{\CC_{\ng}(G)} \times \Z^{\CC_{\fg}(G)}$.

\begin{proof}
Proposition \ref{P:ManyCentralElements} guarantees that $\iota$ is injective, and its image lies in $\Ker(\pi)$. It remains to explain why $\Im(\iota)$ is the whole $\Ker(\pi)$. By Theorem \ref{T:DirectProductDetailed}, $A(G)$ splits as $G^{\Deh} \times \Z^{\CC_{\ng}(G)}$, where $\Z^{\CC_{\ng}(G)}$ is the subgroup generated by the $t_{\OO}$'s with $\OO \in \CC_{\ng}(G)$. But this is precisely $\iota(\Z^{\CC_{\ng}(G)})$. Further, $\iota(\Z^{\CC_{\fg}(G)})$ becomes the subgroup of $G^{\Deh}$ generated by appropriate powers of its generators, which is precisely the kernel of the map $\ G^{\Deh} \twoheadrightarrow \ G$ as described in Theorem \ref{T:DehnAsIntermediate}.
\end{proof}

\begin{xmp}
For $G=S_n$, we get a short exact sequence
\[0 \ \to \ \Z^{P(n)} =\Z^{P(n)-1} \times \Z \ \overset{\iota}{\to} \ A(S_n) \ \overset{\pi}{\to} \ S_n \ \to \ 0,\] 
and the $2$-cocycle $\varphi$ can be computed as follows. Given two permutations $\alpha, \beta \in S_n$, the $\Z^{P(n)-1}$ part of $\varphi(\alpha, \beta)$ is $-c_{\OO(\alpha\beta)}+c_{\OO(\alpha)}+c_{\OO(\beta)}$, where $\OO(\gamma)$ is the conjugacy class of $\gamma \in S_n$, and $c_{\OO}$ is the generator of $\Z^{P(n)-1}=\Z^{\CC_{\ng}(G)}$ corresponding to the class $\OO$ if $\OO \in \CC_{\ng}(G)$, and $0$ otherwise. The $\Z$-part is given by the map
\begin{align*}
\varphi' \colon \hspace*{1cm} S_n^2 &\to \Z,\\
(\alpha,\beta) &\mapsto (-\lg(\alpha\beta)+\lg(\alpha)+\lg(\beta))/2,
\end{align*} 
where $\lg$ is the minimal number of transpositions needed to write a permutation. For instance, for two transpositions $\sigma, \tau$, one has
\[\varphi'(\sigma, \tau)= \begin{cases}  0 & \text { if } \sigma \neq \tau,\\ 1 & \text { if } \sigma = \tau. \end{cases} \]
\end{xmp}

\medskip
For a \cc-group, the map $\pi$ also admits a group-theoretic section $\sigma$, which sends any generator $a \in \Gamma$ to $e_a$, and a general $a \in G$, written in terms of the generators $a_i$ as $a_1^{\alpha_1}\cdots a_s^{\alpha_s}$, to $e_{a_1}^{\alpha_1}\cdots e_{a_s}^{\alpha_s}$. This is a well defined group morphism. Also, $\CC_{\fg}(G)$ is empty for such $G$. We are thus left with the split sequence
\[0 \ \to \ \Z^{\CC_{\ng}(G)} \ \overset{\iota}{\to} \ A(G) \ \overset{\pi}{\underset{\sigma}{\rightleftarrows}} \ G \ \to \ 0,\] 
which recovers the splitting $A(G)\cong G \times \Z^{\CC_{\ng}(G)}$ from Corollary \ref{C:DirectProductCGroup}. In this case the $2$-cocycle $\varphi$ must be a coboundary. Indeed:

\begin{prp}\label{P:Coboundary}
For a \cc-group, one has $\varphi = d f$, for the map
\begin{align*}
f\colon G &\to K(G),\\
 a &\mapsto \sigma(a)^{-1} e_a.
\end{align*}
\end{prp}

\begin{proof}
The map $f$ indeed takes values in $K(G)$, since
\[ \pi(f(a)) = \pi(\sigma(a)^{-1}e_a) = \pi(\sigma(a))^{-1}\pi(e_a) = a^{-1} a = 1.\]
Further,
\begin{align*}
df(a,b)&=f(a)-f(ab)+f(b) = \sigma(a)^{-1} e_a e_{ab}^{-1} \sigma(ab) \sigma(b)^{-1} e_b = \sigma(a)^{-1} e_a e_{ab}^{-1} \sigma(a) e_b \\
& = e_{ab}^{-1} \sigma(a) \sigma(a)^{-1} e_a e_b = e_{ab}^{-1} e_a e_b = \varphi(a,b).
\end{align*}
We used the fact that $\sigma$ is a group morphism, and the centrality of $\sigma(a)^{-1} e_a$.
\end{proof}

\section{Homology}\label{Sec:Hom}

In this section, we will explain how our understanding of the group $A(G)$ can be used, via Eisermann's formula, to compute the second homology group of the conjugation quandle $\Conj(G)$ of a \cb-group $G$. As an illustration, we provide an explicit computation of $H_2(\Conj(S_n))$.

\begin{prp}\label{P:H2Cgroup}
Let $G$ be a \cb-group, $\OO$ its conjugacy class, $a_{\OO} \in \OO$ any representative of this class, and $\OO_1 = \{ 1_G \}$ the class of $1_G$. The corresponding component of formula \eqref{E:Eisermann} then splits as follows:
\begin{equation}\label{E:StabAGvsDehn}
\left(\Stab_{A(G)}(a_{\OO}) \cap \Ker (\varepsilon)\right)_{Ab} \ \cong \left(\Stab_{G^{\Deh}}(a_{\OO})\right)_{Ab} \ \times \ \Z^{\CC_{\ng}(G) \setminus \{\OO_1\}}.
\end{equation}
Further, the map $\pi \colon G^{\Deh}\twoheadrightarrow G$  induces the short exact sequence
\[0 \ \to \ \Z^{\CC_{\fg}(G)} \ \to \ \left(\Stab_{G^{\Deh}}(a_{\OO})\right)_{Ab}  \ \to \ \left(\Stab_{G}(a_{\OO})\right)_{Ab}  \ \to \ 0.\] 
\end{prp}  

Thus, to determine $H_2(\Conj(G))$, one can compute stabilisers in the initial group $G$, lift them to $G^{\Deh}$ via the set-theoretic section $a \mapsto e_a$, and establish all relations among these lifts and the free part $\Z^{\CC_{\fg}(G)}$ in the abelian group $\left(\Stab_{G^{\Deh}}(a_{\OO})\right)_{Ab}$.

\begin{proof}
Recall first that both $G(A)$ and $G^{\Deh}$ act on $G$ on the right via $a \cdot e_b = a \op b$, while $G$ acts by conjugation: $a \cdot b = a \op b = b^{-1}ab$. In particular, for $g \in A(G)$ or $G^{\Deh}$, we have
\begin{equation}\label{E:actions}
a \cdot g \ = \ a \cdot \pi(g) \ = \ a \op \pi(g).
\end{equation}
As a result, the group inclusion $G^{\Deh} \hookrightarrow A(G)$ from \eqref{E:GDinAG} and the map $\pi \colon A(G) \to G$ intertwine these actions on $G$.

Now, in the splitting $A(G) \cong G^{\Deh} \times \Z^{\CC_{\ng}(G)}$ from Theorem \ref{T:DirectProduct}, the part $ \Z^{\CC_{\ng}(G)}$ lies in $\Ker(\pi)$, hence acts trivially on $G$. Consequently, 
\[\Stab_{A(G)}(a_{\OO}) \ \cong \ \Stab_{G^{\Deh}}(a_{\OO}) \times \Z^{\CC_{\ng}(G)}.\]

Next, $\Z^{\CC_{\ng}(G)}$ splits as $\Z^{\CC_{\ng}(G) \setminus \{\OO_1\}} \times \Z^{ \{\OO_1\}}$, and the last factor is generated by the central element $e_{1}$. Since $\varepsilon(e_{1})=1$, one can mod $e_{1}$ out and get the isomorphism
\[\Stab_{\As(Q)}(a_{\OO}) \cap \Ker (\varepsilon) \ \cong  \ \Stab_{G^{\Deh}}(a_{\OO}) \ \times \ \Z^{\CC_{\ng}(G) \setminus \{\OO_1\}}.\]
The second factor is central, hence survives in the abelianisation, yielding \eqref{E:StabAGvsDehn}.

We now focus on the interesting factor $\left(\Stab_{G^{\Deh}}(a_{\OO})\right)_{Ab}$. According to \eqref{E:actions}, an element $g \in G^{\Deh}$ acts trivially on $a \in G$ if and only if $\pi(g)$ and $a$ commute in $G$. In particular, $\Ker(\pi)$ acts trivially on any $a$. Thus the short exact sequence \eqref{E:GDvsG} from Theorem \ref{T:DehnAsIntermediate} induces the short exact sequence
\[0 \ \to \ \Z^{\CC_{\fg}(G)} \ \to \ \Stab_{G^{\Deh}}(a_{\OO}) \ \to \ \Stab_{G}(a_{\OO}) \ \to \ 0.\] 
Since the group $\Z^{\CC_{\fg}(G)}$ is abelian, one goes to the abelianisation and deduces the exact sequence
\[\Z^{\CC_{\fg}(G)} \ \to \ \left(\Stab_{G^{\Deh}}(a_{\OO})\right)_{Ab}  \ \to \ \left(\Stab_{G}(a_{\OO})\right)_{Ab}  \ \to \ 0.\] 
Finally, the leftmost map is injective, since when composed with the map 
\[\left(\Stab_{G^{\Deh}}(a_{\OO})\right)_{Ab} \ \to \  (G^{\Deh})_{Ab} \ \cong  \ \Z^{\CC_{\fg}(G)} \times \Z^{\CC_{\ig}(G)} \] 
induced by the inclusion $\Stab_{G^{\Deh}}(a_{\OO}) \hookrightarrow G^{\Deh}$ it yields multiplication by $k(\OO)$ on the factor corresponding to $\OO \in \CC_{\fg}(G)$, which is injective.
\end{proof}

Finally, we apply these techniques to our favourite example $G=S_n$. Some notations are needed to announce our result. For a partition $\lambda$ of $n$ (written as $\lambda \vdash n$), we call its \emph{support} the set $\Supp(\lambda)$ of sizes of its parts, and \emph{repetition support} the set $\RSupp(\lambda)$ of part sizes repeating in $\lambda$ at least twice. Put
\[\RR(\lambda) \ = \ \begin{cases} \# \RSupp(\lambda)-1& \text { if } \RSupp(\lambda) \text{ contains an odd element},\\
 \# \RSupp(\lambda) & \text{ otherwise}. \end{cases}\]
 Given an even $u \in \N$, let $s(n,u)$ be the number of partitions $\lambda \vdash n$ satisfying the following conditions:
 \begin{itemize}
 \item there is no odd $v \in \RSupp(\lambda)$;
 \item $u \in \Supp(\lambda)$;
 \item a power of $2$ dividing $u$ necessarily divides all other even $u' \in \Supp(\lambda)$;
 \item $u$ is minimal among even numbers satisfying the two preceding properties.
 \end{itemize}
Recall also that $P(n)$ denotes the partition number of $n$.

\begin{thm}\label{T:H2Sn}
The second homology group of the conjugation quandle of the symmetric group can be described as follows:
\begin{align*}
H_2(\Conj(S_n)) \ \cong \ &\Z^{P(n)(P(n)-1)} \times \Z_2^{\sum_{\lambda \vdash n} \RR(\lambda)}\\
& \times  \prod_{\text{odd }u \leq n} \Z_u^{P(n-u)} \times \prod_{\text{even }u \leq n} \left( \Z_u^{P(n-u)-s(n,u)} \times \Z_{\frac{u}{2}}^{s(n,u)} \right).
\end{align*}
\end{thm}

Note the presence of torsion which grows very fast with $n$. In particular, adding a part $1$ to any partition $\lambda \vdash n$, one sees that $H_2(\Conj(S_n))$ is a subgroup of $H_2(\Conj(S_{n+1}))$ for any $n$. We get a much more interesting picture than the group homology case, where $H_2(S_n)$ is simply $0$ for $n<4$ and $\Z_2$ for $n \geq 4$.

For small $n$, our formulas yield 
\begin{align*}
H_2(\Conj(S_3)) \ \cong \ & \Z^6 \times \Z_3,\\
H_2(\Conj(S_4)) \ \cong \ & \Z^{20} \times \Z_2^3 \times \Z_3,\\
H_2(\Conj(S_5)) \ \cong \ & \Z^{42} \times \Z_2^3 \times \Z_3^2 \times \Z_5,\\
H_2(\Conj(S_6)) \ \cong \ & \Z^{110} \times \Z_2^4 \times \Z_3^4 \times \Z_4^2 \times \Z_5,\\
H_2(\Conj(S_7)) \ \cong \ & \Z^{210} \times \Z_2^7 \times \Z_3^6 \times \Z_4^2 \times \Z_5^2 \times \Z_7.
\end{align*}

\begin{proof}
Take the conjugacy class $\OO_{\lambda}$ of $S_n$ indexed by a partition $\lambda = (\lambda_1, \ldots, \lambda_k) \vdash n$, and its representative $a_{\lambda}=c_1 \cdots c_k$ written as a product of cycles $c_i$ of length $\lambda_i$ with disjoint support. The proposition above yields the isomorphism,
\[\left(\Stab_{A(S_n)}(a_{\lambda}) \cap \Ker (\varepsilon)\right)_{Ab} \ \cong \left(\Stab_{S_n^{\Deh}}(a_{\lambda})\right)_{Ab} \ \times \ \Z^{P(n)-2},\]
and the short exact sequence
\begin{equation}\label{E:SESSn}
0 \ \to \ \Z \ \overset{\iota}{\to} \ \left(\Stab_{S_n^{\Deh}}(a_{\lambda})\right)_{Ab}  \ \overset{\widetilde{\pi}}{\to} \ \left(\Stab_{S_n}(a_{\lambda})\right)_{Ab}  \ \to \ 0.
\end{equation}
The subgroup $\Stab_{S_n}(a_{\lambda})$ of $S_n$ is generated by the cycles $c_i$, and, for each $v \in \RSupp(\lambda)$ appearing in $\lambda$ $k_{v} \geq 2$ times, by a subgroup isomorphic to $S_{k_{v}}$, which permutes the supports of the $v$-cycles. In the abelianisation, the cycle $c_i$ with $\lambda_i \notin \RSupp(\lambda)$ yields a factor $\Z_{\lambda_i}$, and the $\Z_{v}^{k_{v}} \rtimes S_{k_{v}}$ component of $\Stab_{S_n}(a_{\lambda})$ corresponding to a $v \in \RSupp(\lambda)$ yields $\Z_{v} \times \Z_2$. Hence
\begin{equation}\label{E:StabSnAb}
\left(\Stab_{S_n}(a_{\lambda})\right)_{Ab} \ \cong \ \prod_{u \in \Supp(\lambda)} \Z_{u} \times \prod_{v \in \RSupp(\lambda)} \Z_2.
\end{equation}

 For each $u \in \Supp(\lambda)$, $u \neq 1$, take a generator $c_i=c_{i;1}  \cdots c_{i;u-1}$ of the corresponding component $\Z_{u}$ written as a product of $u-1$ transpositions, and put $e_{u}:= e_{c_{i;1}} \cdots e_{c_{i;u-1}}$ (we abusively use the same notations in a group and in its abelianisation). Finally, for $v \in \RSupp(\lambda)$, denote by $d_{v}$ any element of $S_n$ permuting two cycles of length $v$ in $a_{\lambda}$. It generates the copy of $\Z_2$ in $\Z_{v} \times \Z_2$. It can be written as a composition of $v$ transpositions with disjoint support, $d_{v} = d_{v;1}\cdots d_{v;v}$. Let $f_{v}$ be its lift $e_{d_{v;1}} \cdots e_{d_{v;v}}$ in $S_n^{\Deh}$. By \eqref{E:SESSn}, the abelian group $\left(\Stab_{S_n^{\Deh}}(a_{\lambda})\right)_{Ab}$ is generated by the $e_{u}$, the $f_{v}$, and a generator $t=e_{\tau}^2$ of $\Im(\iota)$, independent of the choice of transposition $\tau$. We need to determine all relations between these generators. The most obvious ones concern the powers of these generators:
\begin{align}
e_{u}^{u} &= t^{\frac{u(u-1)}{2}}, & u \in \Supp(\lambda), u \neq 1\label{E:H2Snpowers1}\\
f_{v}^2 &= t^{v}, & v \in \RSupp(\lambda).\label{E:H2Snpowers2}
\end{align} 
Indeed, we chose (the minimal) exponents such that these powers of the generators $e_{u}$ and $f_{v}$ lie in $\Ker(\widetilde{\pi})$ and thus coincide with some powers of $t$, which are determined by computing degrees. Further, consider a relation
\begin{equation}\label{E:H2SnRelation}
\prod_{u \in \Supp(\lambda)} e_{u}^{\alpha_{u}} \prod_{v \in \RSupp(\lambda)} f_{v}^{\beta_{v}} \ = \ t^{\gamma}, \hspace*{1cm} \alpha_{u},\beta_{v}, \gamma \in \Z,
\end{equation}
in $\left(\Stab_{S_n^{\Deh}}(a_{\lambda})\right)_{Ab}$. Applying $\widetilde{\pi}$ and comparing the image in \eqref{E:StabSnAb}, one sees that each $\alpha_{u}$ should be a multiple of $u$, and each $\beta_{v}$ a multiple of $2$. Thus relations \eqref{E:H2Snpowers1}-\eqref{E:H2Snpowers2} can be used to turn \eqref{E:H2SnRelation} into an equality between two powers of $t$, which can only be trivial since $t$ generates a copy of $\Z$.

Let us now try to simplify relations \eqref{E:H2Snpowers1}-\eqref{E:H2Snpowers2} by variable substitution.
\begin{itemize}
\item For odd $u$, put $E_{u}:=e_{u}t^{-\frac{(u-1)}{2}}$. Relation \eqref{E:H2Snpowers1} then becomes $E_{u}^{u}=1$.
\item For even $u$, put $E_{u}:=e_{u}t^{-\frac{(u-2)}{2}}$. Relation \eqref{E:H2Snpowers1} then becomes $E_{u}^{u}=t^{\frac{u}{2}}$.
\item For odd $v$, put $F_{v}:=f_{v}t^{-\frac{(v-1)}{2}}$. Relation \eqref{E:H2Snpowers2} then becomes $F_{v}^{2}=t$.
\item For even $v$, put $F_{v}:=f_{v}t^{-\frac{v}{2}}$. Relation \eqref{E:H2Snpowers2} then becomes $F_{v}^{2}=1$.
\end{itemize}

\begin{description}
\item[Case 1] There exists an odd $v_0 \in \RSupp(\lambda)$. Put $\tau:= F_{v_0}$. The variable $t=\tau^2$ can then be omitted. For even $u$, put $\overline{E}_{u}:=E_{u} \tau^{-1}$. Relation \eqref{E:H2Snpowers1} then becomes $\overline{E}_{u}^{u}=1$. Also, for odd $v \neq v_0$, put $\overline{F}_{v}:=F_{v}\tau^{-1}$. Relation \eqref{E:H2Snpowers2} then becomes $\overline{F}_{v}^{2}=1$. All relations are thus transformed into the triviality of certain powers of generators. Put together, they yield
\[\left(\Stab_{S_n^{\Deh}}(a_{\lambda})\right)_{Ab} \ \cong \ \left( \prod_{u \in \Supp(\lambda)} \Z_{u} \right) \times \Z_2^{\# \RSupp(\lambda) -1} \times \Z.\]
\item[Case 2] There is no odd $v \in \RSupp(\lambda)$. Then one has a factor $\Z_{u}$ for each odd $u \in \Supp(\lambda)$, and a factor $\Z_2$ for each (necessarily even) $v \in \RSupp(\lambda)$. The remaining part $G_1$ has as generators the $E_{u}$ for even $u \in \Supp(\lambda)$, together with the generator $t$; they are subject to the relations $E_{u}^{u}=t^{\frac{u}{2}}$ only. Denote the values of $u$ concerned as $2h_1, \ldots, 2h_s,\ h_i \in \N$, arranged so that $h_1$ is the minimal among the $h_i$ having the minimal $v_2(h_i)$ (where $v_2(h_i)$ is the maximal power of $2$ dividing $h_i$). It will be denoted by $m(\lambda):=h_1$. Let us show that the abelian group $G_1$ isomorphic to
\[G_2 \ := \ \Z_{h_1} \times \Z_{2h_2} \times \cdots \times \Z_{2h_s} \times \Z.\]
Our defining relations for $G_1$ can be arranged into the following matrix:
	\[M_1 = \left\lbrack \begin{array}{ccccccccc}
	2h_1 && && && && -h_1 \\
	& 2h_2 &&& && && -h_2 \\
	 && && \cdots && && \cdots \\
	 && && && 2h_s && -h_s 	
	\end{array} \right\rbrack.\]
For $G_2$, one gets
	\[M_2 = \left\lbrack \begin{array}{ccccccccc}
	h_1 && && && && 0 \\
	& 2h_2 &&& && && 0 \\
	 && && \cdots && && 0 \\
	 && && && 2h_s	&& 0 
	\end{array} \right\rbrack.\]
By the existence and uniqueness of the Smith normal form, our two abelian groups are isomorphic if and only if, for all $i$, the matrices above share the same greatest common divisors (gcd) of the determinants of all their $i\times i$ minors. For $M_1$, such a determinant is a product of factors $2h_q$ for distinct $q$, possibly times a factor $h_r$. So the gcd of such determinants for a fixed $i$ is $2^{i-1}$ times the gcd of the products $h_{q_1} \cdots h_{q_i}$ for all $1 \leq q_1 < \cdots < q_i \leq n$. For $M_2$, such a determinant is a product of factors $2h_q$ with distinct $q>1$, possibly times the factor $h_1$. For any prime $p>2$, their gcd will contain the same power of $p$ that the gcd for $M_1$. For $p=2$, one of the products containing the minimal power of $2$ for both matrices is $h_1*2h_2*\cdots*2h_i$. Hence the two gcd coincide. Summarising, one gets
\begin{equation}\label{E:H2DifficultCase}
\left(\Stab_{S_n^{\Deh}}(a_{\lambda})\right)_{Ab} \ \cong \ \Z_{\frac{m(\lambda)}{2}} \times \left( \prod\limits_{\substack{u \in \Supp(\lambda) \\ u \neq m(\lambda)}} \Z_{u} \right) \times \Z_2^{\# \RSupp(\lambda) } \times \Z,
\end{equation}
where the first term is omitted if $\lambda$ contains no even parts.
\end{description}

By \eqref{E:Eisermann}, the terms we have just computed need to be summed for all $\lambda \vdash n$.
\begin{itemize}
\item The free part of each of these $P(n)$ terms is $\Z \times \Z^{P(n)-2} \cong \Z^{P(n)-1}$. Hence the free part of the sum is $\Z^{P(n)(P(n)-1)}$, as predicted by \cite{EtGrRackCohom,LiNel}.
\item Summing the $\Z_2$ factors corresponding to the generators $F_v$, one gets
\[\prod_{\lambda \vdash n} \Z_2^{\RR(\lambda)} \cong \Z_2^{\sum_{\lambda \vdash n} \RR(\lambda)}.\]
\item Each odd $u \in \Supp(\lambda)$ brings a factor $\Z_u$. Since there are $P(n-u)$ partitions $\lambda \vdash n$ having $u \in \Supp(\lambda)$, the overall contribution to $H_2$ is $\Z_u^{P(n-u)}$. 
\item The most delicate case is that of even $u \in \Supp(\lambda)$, since they bring either a factor $\Z_u$, or, when $\lambda$ has no repeating odd parts and $u=m(\lambda)$, a factor $\Z_{\frac{u}{2}}$. The sum of these factors is 
\[\Z_u^{P(n-u)-s(n,u)} \times \Z_{\frac{u}{2}}^{s(n,u)}.\]
\end{itemize}
The sum of these four parts yields the announced result.
\end{proof}

\begin{rmr}
The proof easily adapts to a computation of $H_2(\Transp_n)$ for the $1$-orbit quandle $\Transp_n$ of all transpositions in $S_n$, endowed with the conjugation operation. Since $\As(\Transp_n)=S_n^{\Deh}$, one gets
\[H_2(T_n)\ \cong\ \left(\Stab_{S_n^{\Deh}}(a_{(2,1,\ldots,1)}) \cap \Ker (\varepsilon)\right)_{Ab}.\]
This is generated by elements $t$, $e_2$, and, if $n \geq 4$, $f_1$, subject to relations $e_2^2=f_1^2=t$. Since $\varepsilon(e_2) = \varepsilon(f_1)=1$, intersection with $\Ker (\varepsilon)$ is zero for $n \leq 3$, and is generated by the element $e_2f_1^{-1}$ of order $2$ otherwise. One obtains
\[H_2(T_n)\ \cong\ \begin{cases} 0, & n \leq 3,\\
\Z_2, & n \geq 4,
\end{cases}\]
which recovers a computation from \cite{GarIglVen}.
\end{rmr}

\begin{rmr}
Note that in all the cases from the proof, the torsion of  $\Stab_{S_n^{\Deh}}(a_{\lambda})$ differs from its finite quotient $\Stab_{S_n}(a_{\lambda})$, which behaves more regularly, by at most one factor (one factor $\Z_2$ disappears, or one factor $\Z_{2h}$ becomes $\Z_h$). This suggests that an alternative organisation of terms in the most difficult case \eqref{E:H2DifficultCase} could probably lead to a more compact final formula, which we were unable to find. 
\end{rmr}

\subsection*{Acknowledgments}
The author is grateful to Carsten Dietzel, Eddy Godelle, Neha Nanda, and Markus Szymik for helpful remarks and pointers to the literature. 

\bibliographystyle{alpha}
\bibliography{refs}
\end{document}